\documentclass[a4paper,12pt]{amsart}
\usepackage{amsmath,amsthm,amssymb}
\usepackage[arrow,matrix]{xy}
\theoremstyle{plain}
\newtheorem{theorem}{Theorem}[section]
\newtheorem{proposition}[theorem]{Proposition}
\newtheorem{lemma}[theorem]{Lemma}

\newtheorem{condition}[theorem]{\rm Condition}

\theoremstyle{definition}

\newtheorem{definition}[theorem]{\rm Definition}
\newtheorem{remark}[theorem]{\rm Remark}

\newtheorem{example}[theorem]{\rm Example}

\numberwithin{equation}{section}

\newcommand{\field}[1]{\mathbb{#1}}
\newcommand{\C}{\field{C}}

\newcommand{\F}{\field{F}}

\newcommand{\Q}{\field{Q}}
\newcommand{\R}{\field{R}}
\newcommand{\Z}{\field{Z}}

 \DeclareMathOperator{\Hom}{Hom}

\DeclareMathOperator{\sgn}{sgn}

\DeclareMathOperator{\interior}{int}

 \DeclareMathOperator{\id}{id}

\def\Sd{S_d}
\def\HomSG{\Hom(\Sd,G_d)}
\def\HomS1n{\Hom(S^1,(S^1)^n)}

\newcommand{\mathscr}{}

\def\higher{generalized\ }
\def\Higher{Generalized\ }

\title[Quasitoric manifolds]{Quasitoric manifolds over a product of simplices}
\author[S. Choi]{Suyoung Choi}
\address{Department of Mathematical sciences, Korea Advanced
Institute of Science and Technology, 335 Gwahangno, Yuseong-gu, Daejeon 305-701, Republic of Korea} \email{choisy@kaist.ac.kr}
\author[M. Masuda]{Mikiya Masuda}
\address{Department of Mathematics, Graduate School of Science, Osaka City University} \email{masuda@sci.osaka-cu.ac.jp}
\author[D.Y. Suh]{Dong Youp Suh} \address{Department of Mathematical sciences, Korea Advanced
Institute of Science and Technology, 335 Gwahangno, Yuseong-gu, Daejeon 305-701, Republic of Korea} \email{dysuh@math.kaist.ac.kr}
\thanks{The first author was partially supported by the second stage of Brain Korea 21 project, KAIST in 2007,
the second author was partially supported by Grand-in-Aid for Scientific Research 19204007, and the
third author was partially supported by the SRC program of Korea Science and Engineering Foundation R11-2007-035-02002-0.
}
\dedicatory{Dedicated to Professor Takao Matumoto on his sixtieth birthday}
\date{\today}

\begin{document}

\begin{abstract}
A quasitoric manifold (resp. a small cover) is a $2n$-dimensional
(resp. an $n$-dimensional) smooth closed manifold with an effective
locally standard  action of $(S^1)^n$ (resp.
$(\mathbb Z_2)^n$) whose orbit space is combinatorially an
$n$-dimensional simple convex polytope $P$. In this paper we
study them when $P$ is a product of simplices.
A \higher Bott tower over $\F$, where $\F=\C$ or $\R$,
is a sequence of projective bundles of the Whitney sum of $\F$-line bundles
starting with a point.
Each stage of the tower over $\F$, which we call a \higher Bott manifold,
provides an example of quasitoric manifolds (when $\F=\C$) and small covers
(when $\F=\R$) over a product of simplices.
It turns out that every small cover over a product of simplices
is equivalent (in the sense of Davis and Januszkiewicz \cite{DJ})
to a \higher Bott manifold. But this is not the case
for quasitoric manifolds and we show that a quasitoric manifold over
a product of simplices
is equivalent to a \higher Bott manifold if and only if it admits
an almost complex structure left invariant under the action.
Finally, we show that a quasitoric manifold $M$ over a product of simplices
is homeomorphic to a \higher
Bott manifold if $M$ has the same cohomology ring as
a product of complex projective spaces with $\Q$ coefficients.

\end{abstract}

\maketitle

\section{Introduction}\label{Introduction}

Toric varieties in algebraic geometry and Hamiltonian torus actions
on symplectic manifolds exhibit fascinating  relations between the
geometry of algebraic varieties or smooth manifolds and the
combinatorics of their orbit spaces. Considering the success of
toric theory, it is natural to generalize them to the topological
category, and a monumental development in this direction was
obtained by the work of Davis and Januszkiewicz in \cite{DJ}. They
defined a topological generalization of toric variety by the name of
\lq\lq toric manifold", which is a $2n$-dimensional closed manifold
$M$ with a locally standard action of $n$-torus
$G=(S^1)^n$ whose orbit space is combinatorially an $n$-dimensional
simple convex polytope $P$. In this case $M$ is said to be  a \lq\lq
toric manifold" over $P$. They also defined a $\Z_2$-analogue
of a \lq\lq toric manifold" called a small cover, which is an
$n$-dimensional manifold with an effective action of the $\mathbb
Z_2$-torus of rank $n$ with an $n$-dimensional simple polytope as
the orbit space.

Unfortunately the term \lq\lq toric manifolds" is already
well-established among algebraic geometers as \lq\lq non-singular
toric variety". Moreover there are \lq\lq toric manifolds" (in the
sense of Davis an Januszkiewicz) which are not algebraic varieties,
for example $\mathbb CP^2\sharp\, \mathbb CP^2$. Because of this
reason Buchstaber and Panov introduced the term \lq\lq quasitoric
manifold" as an alias for Davis and Januszkiewicz's \lq\lq toric
mani\-fold" in \cite{BP}. In this paper we adopt Buchstaber and
Panov's \lq\lq quasitoric manifold" instead of \lq\lq toric
manifold". We refer the reader to Chapter 5 of \cite{BP} for an
excellent exposition on quasitoric manifolds including their
comparison with (compact non-singular) toric varieties.

This paper is motivated by the work \cite{MP} which investigates quasitoric
manifold over a cube.  A cube is a product of 1-simplices.
We take a product of simplices as the simple polytope
$P$ and describe quasitoric manifolds and small covers over $P$ in terms
of matrices with vectors as entries.
A typical example of quasitoric manifolds or small covers over a product of
simplices appears in a sequence of projective bundles
\[
\xymatrix{ B_m\ar[r]^-{\pi_m}& B_{m-1}\ar[r]^-{\pi_{m-1}}&
\cdots\ar[r]^-{\pi_2}& B_1\ar[r]^-{\pi_1}& B_0}=\{\textrm{a point}\},
\]
where $B_i$ for $i=1,\dots,m$ is the projectivization of the Whitney sum of
$n_i+1$ $\F$-line bundles over $B_{i-1}$ ($\F=\C$ or $\R$).
Grossberg-Karshon \cite{gr-ka94} considered the sequence above when $\F=\C$
and $n_i=1$ for any $i$, and they named it a {\em Bott tower}.
Motivated by this, we call the sequence above
a {\em  \higher Bott tower} (over $\F$).
The $j$-stage $B_j$ of the tower provides
a quasitoric manifold (when $\F=\C$) and a small cover (when $\F=\R$)
over $\prod_{i=1}^j\Delta^{n_i}$ where $\Delta^{n_i}$ is the $n_i$-simplex.
We call each $B_j$ a {\em \higher Bott manifold} (over $\F$) and
especially call it a {\em Bott manifold} when the tower is a Bott tower.
It turns out that any small cover over a product of simplices
is equivalent (in particular, homeomorphic)
to a \higher Bott manifold (over $\R$) (see Remark~\ref{difference})
but this is not the case for quasitoric manifolds.  We give a necessary
and sufficient condition for a quasitoric manifold over a product of simplices
to be equivalent to a \higher Bott manifold (over $\C$)
(see Theorem~\ref{theorem:complex Bott manifold}),
where a part of the statement is a particular case of
\cite[Theorem 6]{Do}.

This paper is organized as follows.
In Section 2 we recall general facts on quasitoric manifolds and small covers
over a simple polytope.  From Section 3 we restrict our concern to
a product of simplices as the simple polytope and treat only quasitoric
manifolds because small covers can be treated similarly.
In Section 3 we introduce some notation needed for later discussion
and associate a matrix with vectors as entries to a quasitoric manifold
over a product of simplices.
In Section 4 we describe quasitoric manifolds over a product of simplices as
the orbit space of a product of odd dimensional
spheres by some free torus action. This is done in \cite{gr-ka94}
and \cite{ci-ra05} when the orbit space is a product of 1-simplices, that is,
a cube.
The association of the matrix with vectors as entries to a quasitoric
manifold over a product of simplices depends on the order of the
product of the simplices.  We discuss this in Section 5.
\Higher Bott towers are introduced in Section 6 and \higher Bott manifolds are
characterized among quasitoric manifolds over a product of simplices
(Theorem~\ref{theorem:complex Bott manifold}).
In Section 7 we explicitly describe the cohomology ring of a quasitoric
manifold over a product of simplices and prove in Section 8 that such
a quasitoric manifold is homeomorphic to a \higher Bott manifold if it has the
same cohomology ring as a product of complex
projective spaces with $\Q$ coefficients.

\section{General facts} \label{sect:2}

An $n$-dimensional convex polytope $P$ is said to be
{\em simple} if precisely $n$
facets (namely codimension-one faces of $P$) meet at each vertex.
Equivalently, $P$ is simple if the dual of the boundary complex
$\partial P$ of $P$ is a simplicial complex. It is clear that every
simplex is simple and a product of simple convex polytopes is simple.
Therefore a product of simplices is simple.

Let $d=1$ or $2$.  We denote by $\Sd$ an order two group $S^0$ when $d=1$ and
a circle group $S^1$
when $d=2$, and by $G_d$ a group isomorphic to $(\Sd)^n$.
A $dn$-dimensional smooth $G_d$-manifold $M_d$ with a projection
$\pi\colon M_d\to P$ is called a {\em small cover} (when $d=1$) and
a {\em quasitoric manifold} (when $d=2$) over an  $n$-dimensional simple
convex polytope $P$ if $M_d$ is locally isomorphic to a faithful real
$dn$-dimensional
representation of $G_d$ and each fiber of $\pi$ is a $G_d$-orbit.
The orbit space $M_d/G_d$ can be identified with $P$.
Two quasitoric manifolds or small covers $\pi\colon M_d\to P$ and
$\pi'\colon M_d'\to P$
are {\em equivalent (in the sense of Davis and Januszkiewicz)} if
there is a homeomorphism $f\colon M_d\to M_d'$ covering the identity on $P$
and an automorphism $\theta:G_d\to G_d$ such that $f$ satisfies
$\theta$-equivariance, i.e., $f(gm)=\theta(g)f(m)$ for all $m\in
M_d$ and $g\in G_d$. Note that the equivalence is neither weaker nor
stronger than $G_d$-homeomorphism, because any $G_d$-homeomorphism
must satisfy $\theta$-equivariance with $\theta=\id$, but it may not
cover the identity on the orbit space.

Let $\pi\colon M_d\to P$ be a small cover or a quasitoric manifold and
let $\mathcal F$ be the set of facets of $P$.
If $F\in \mathcal F$, then the isotropy subgroup of a point $x\in
\pi^{-1}(\interior F)$ is independent of the choice of $x$, and is a
rank-one subgroup $G_d(F)$ of $G_d$.
The group $\HomSG$ of homomorphisms from $\Sd$ to $G_d$ is isomorphic
to $(R_d)^n$ where $R_d$ is $\mathbb Z/2$
when $d=1$ and $\mathbb Z$ when $d=2$.
Each rank-one subgroup
of $G_d$ corresponds uniquely (up to sign) to a primitive vector
of $\HomSG$ which generates a rank-one direct summand of
$\HomSG$. Therefore every $M_d$ defines what is called the {\em characteristic
function} of $M_d$
$$\lambda\colon \mathcal F \to \HomSG$$
such that the image of $F\in \mathcal F$ is a primitive vector of
$\HomSG$ corresponding to the rank-one subgroup $G_d(F)$.
When $d=1$, such a primitive vector is unique for each $F$,
but sign ambiguity arises when $d=2$. This sign ambiguity can be resolved
if an omniorientation (see \cite{BP}) is assigned to a quasitoric manifold
$M_d$, in particular if $M_d$ admits an almost complex structure left
invariant under the action (see Lemma 1.5 and 1.10 of \cite{Ma}).
In any case, the characteristic function $\lambda$ of $M_d$
must satisfy the following condition, see \cite{DJ}.

\begin{condition}\label{condition}
If n facets $F_1,\ldots, F_n$ of $P$ intersect at a vertex, then their
images $\lambda(F_1),\ldots,\lambda(F_n)$ must form a basis of
$\HomSG$.
\end{condition}

Conversely, for a function $\lambda\colon \mathcal F\to \HomSG$
satisfying Condition~\ref{condition}, there exists a unique
(up to equivalence) small cover
(when $d=1$) and quasitoric manifold (when $d=2$) with $\lambda$ as the
characteristic function, see \cite{DJ} or \cite{bu-pa-ra06} for details.
Therefore in order to
classify all small covers or quasitoric manifolds over a simple convex
polytope $P$, it is necessary and sufficient to understand the functions
$\lambda$ satisfying Condition~\ref{condition}.

Let $F_1,\dots,F_k$ be the all facets of $P$ and let $\omega_1,\dots,\omega_k$
be the indeterminates corresponding to the facets.
Then it is shown in \cite{DJ} that the equivariant cohomology ring
$H^{\ast}_{G_d}(M_d;R_d)$ is
the face ring (or the Stanley-Reisner ring) of $P$ with $R_d$
coefficient as graded rings, that is,
\begin{equation}\label{equation:equivariant cohomology}
H^{\ast}_{G_d}(M_d;R_d)=R_d[\omega_1,\ldots, \omega_k]/I,
\end{equation}
where the degree of $\omega_i$ is $d$ for each $i$
and $I$ is the homogeneous ideal of the polynomial ring $R_d[\omega_1,\ldots,
\omega_k]$ generated by all
square-free monomials of the form $\omega_{i_1}\cdots \omega_{i_s}$ such
that the intersection of the corresponding facets $F_{i_1}, \ldots, F_{i_s}$
is empty.

We choose a basis of $\HomSG$ and identify $\HomSG$ with $(R_d)^n$.
We form a $k\times n$ matrix whose $i$-th row is $\lambda(F_i)\in (R_d)^n$,
i.e.,
\begin{equation} \label{equation:lambda}
(\lambda_{ij})=\left(\begin{array}{c}
\lambda(F_1)\\
\vdots\\
\lambda(F_k)\end{array}\right).
\end{equation}
Let $\lambda_j=\lambda_{1j}\omega_1+\cdots+\lambda_{kj}\omega_k$, and let $J$ be the ideal of  $R_d[\omega_1,\ldots, \omega_k]$
generated by $\lambda_j$ for $j=1,\ldots, n$.
Then we have
\begin{equation}\label{equation:cohomology ring}
H^{\ast}(M_d;R_d)=R_d[\omega_1,\ldots, \omega_k]/(I+J).
\end{equation}

\begin{remark}
In general it would be natural to
use a \emph{column} vector to express $\lambda(F_i)$
(see \cite{BP}), but then, as noticed in \cite{MP}, we need to
take a transpose of a matrix at some point to adjust our description to
the notation used in \cite{ci-ra05} and \cite{gr-ka94}. Therefore
we will use a \emph{row} vector to express $\lambda(F_i)$ in this paper.
\end{remark}

As is seen above, most of the arguments for quasitoric manifolds work
for small covers with $S^1$ and $\Z$ replaced by $S^0$ and $\Z/2$
respectively.
In fact, the study of small covers is a bit simpler than that of
quasitoric manifolds in our case.
So we shall treat only quasitoric manifolds throughout this paper.
The main difference between quasitoric manifolds and small covers
in our arguments is stated in Remark~\ref{difference}, so that the arguments
after Section~\ref{sect:cohom} are unnecessary for small covers.

\section{Vector matrices}
\label{sect:3} From now on, we take
$$P=\prod_{i=1}^m\Delta^{n_i}, \qquad\text{with}\quad \sum_{i=1}^m n_i=n,$$
where $\Delta^{n_i}$ is the $n_i$-simplex for $i=1,\ldots m$.
Let $\{ v^i_0,\ldots,v^i_{n_i}\}$ be the set of vertices of the
simplex $\Delta^{n_i}$. Then each vertex of $P$ is the product of
vertices of $\Delta^{n_i}$'s for $i=1,\ldots,m$, hence the set of vertices of $P$ is
$$\{v_{j_1\ldots j_m}=v^1_{j_1}\times\cdots\times v^m_{j_m}\mid
0\le j_i\le n_i\}.$$
Each facet of $P$ is the product of a codimension-one face  of one of
$\Delta^{n_i}$'s and the remaining simplices. Therefore the set of
facets of $P$ is
$$\mathcal F=\{F^i_{k_i}\mid 0\le k_i\le n_i \quad  i=1,\ldots, m\}$$
where
$F^i_{k_i}=\Delta^{n_1}\times\cdots\times\Delta^{n_{i-1}}\times
f^i_{k_i}\times\Delta^{n_{i+1}}\times\cdots\times\Delta^{n_m}$, and
$f^i_{k_i}$ is the codimension-one face of the simplex
$\Delta^{n_i}$ which is opposite to the vertex $v^i_{k_i}$. Hence
there are $\sum_{i=1}^m(n_i+1)=n+m$ facets in $P$. Since $P$ is
simple, exactly $n$ facets meet at each vertex. Indeed, at each
vertex $v_{j_1\ldots j_m}$ of $P$ all $n$ facets in $\mathcal
F-\{F^i_{j_i}\mid i=1,\ldots , m\}$ intersect, in particular,
the $n$ facets
in the set $$\mathcal F - \{F_0^i\mid i=0,\ldots,
m\}=\{F^1_1,\ldots,F^1_{n_1},\ldots,F^m_1\ldots,F^m_{n_m}\}$$
intersect at the vertex $v_{0\ldots0}$.

Let $\lambda\colon \mathcal F\to \HomS1n$ be the characteristic function of a
quasitoric manifold over $P$. By Condition~\ref{condition}, $n$ vectors
\begin{equation} \label{basis}
\lambda(F^1_1),\ldots,\lambda(F^1_{n_1}),\ldots,\lambda(F^m_1)\ldots,
\lambda(F^m_{n_m})
\end{equation}
form a basis of $\HomS1n$ and we identify $\HomS1n$ with $\Z^n$ through
this basis.  Then the vectors in (\ref{basis}) correspond
to the standard basis elements
$$\mathbf e_1=(1,0,\ldots,0), \ldots,  \mathbf e_n= (0,\ldots,0,1)$$
in the given order. For the remaining $m$ facets $F^i_0$, we set
\[
\lambda(F^i_0)=\mathbf a_i\in \Z^n \quad\text{for  $i=1,\ldots, m$.}
\]
In this way, to the characteristic function $\lambda$ of a quasitoric manifold
over $P$ we have a corresponding $m\times n$ matrix
$$ A=\left(\begin{array}{c}
    \mathbf a_1\\
    \vdots\\
    \mathbf a_m
    \end{array} \right ),\ \textrm{where}\ \mathbf a_i\in \Z^n.
$$

Each row vector $\mathbf a_i$ can be written as
\begin{eqnarray*}
\mathbf a_i & =&(\mathbf a^1_i,\ldots \mathbf a^j_i,\ldots,\mathbf a^m_i)\\
&=&([a^1_{i 1},\ldots,a^1_{in_1}],\ldots,
[a^j_{i1},\ldots,a^j_{in_j}],\ldots,[a^m_{i1}\ldots,a^m_{in_m}])\end{eqnarray*}
where $\mathbf a^j_i=[a^j_{i1},\ldots,a^j_{in_j}]\in \Z^{n_j}$
for $j=1,\ldots, m$. Therefore we may write
\begin{equation}\label{equation:matrix A}
\begin{array}{l l l}
A&=& \left( \begin{array}{c}
    \mathbf a_1\\
    \vdots\\
    \mathbf a_m
    \end{array} \right)
    =
    \left( \begin{array}{c c c}
    \mathbf a^1_1 & \ldots & \mathbf a_1^m\\
    \vdots & \ldots & \vdots\\
    \mathbf a_m^1 & \ldots & \mathbf a_m^m
    \end{array}\right)\\
& = &   \left(\begin{array}{c c c c c c c}
    a^1_{11} & \ldots & a^1_{1n_1} & \ldots & a^m_{11} & \ldots & a^m_{1n_m}\\
    \vdots &&&&&&\vdots\\
    a^1_{m1} & \ldots & a^1_{mn_1} & \ldots & a^m_{m1} & \ldots &
    a^m_{mn_m}
    \end{array}\right)
    \end{array}
\end{equation}
with $\mathbf a^j_i\in \Z^{n_j}$ for all $i=1,\ldots,m$.
In other words, the $m\times n$ matrix $A$ can be viewed as an $m\times m$
matrix whose entries in the $j$-th column are vectors in $\Z^{n_j}$. >From now on, we shall view the matrix $A$ this way and call it
a {\em vector matrix}.

Since the characteristic function $\lambda$ satisfies
Condition~\ref{condition}, we need to translate this into a condition on the
corresponding matrix $A$. For this let us fix some more notation.
For given $1\le k_j\le n_j$ with $j=1,\ldots, m$, let $A_{k_1\ldots
k_m}$ be the  $m\times m$ submatrix of $A$ whose $j$-th column is
the $k_j$-th column of the $m\times n_j$ matrix
$$\left(\begin{array}{c}\mathbf a^j_1\\ \vdots \\ \mathbf a^j_m\end{array}
\right)=
\left(\begin{array}{c c | c | c c}
a^j_{11}&\ldots & \overline{ a^j_{1k_j} }& \ldots &  a^j_{1n_j} \\
\vdots& &\vdots &  & \vdots\\
a^j_{m1}&\ldots &\underline{a^j_{mk_j}} & \ldots & a^j_{mn_j}\\
\end{array}\right).$$
Thus
$$A_{k_1\ldots k_m}=\left(\begin{array}{c c c}
a^1_{1k_1} & \ldots & a^m_{1k_m}\\
\vdots & & \vdots\\
a^1_{mk_1}&\ldots & a^m_{mk_m}
\end{array}\right).$$

\begin{example} \label{example:triangular cylinder}
Let $P=\Delta^2\times \Delta^1$ be a triangular cylinder.  Let $\{v^1_0,
v^1_1, v^1_2\}$ be the vertices of $\Delta^2$ and $\{v^2_0, v^2_1\}$
the vertices of $\Delta^1$. Then
$$\{v_{00}, v_{10}, v_{20}, v_{01},v_{11}, v_{21}\}$$
is the vertex set of $P$ where $v_{ij}=v^1_i\times
v^2_j$. The set of facets of $P$ is
$$\{ F^1_0, F^1_1, F^1_2, F^2_0,F^2_1\}$$
where $F^1_i=f^1_i\times \Delta^1$ for $i=0,1,2$ are the side
rectangles and $F^2_j=\Delta^2\times f^2_j$ for $j=0,1$ are the top
and bottom triangles. The characteristic function $\lambda\colon
\mathcal F\to \mathbb Z^3$ is assigned as
follows:
$$\lambda(F^1_0)=\mathbf a_1,\ \lambda(F^1_1)=\mathbf e_1,\
\lambda(F^1_2)=\mathbf e_2$$
$$ \lambda(F^2_0)=\mathbf a_2,\ \lambda(F^2_1)=\mathbf e_3.$$

The corresponding $2\times 3$ matrix $A$ is

\begin{eqnarray*}
A&=  &\left(\begin{array}{c} \mathbf a_1\\ \mathbf a_2\end{array} \right) \\
&=   &\left(\begin{array}{cc}
\mathbf a^1_1 & \mathbf a^2_1\\
\mathbf a^1_2 & \mathbf a^2_2\end{array}\right)\quad \textrm{as a $2\times 2$ vector matrix}\\
&=   &\left(
\begin{array}{cc}
 a^1_{11} & a^1_{12}\\
 a^1_{21} & a^1_{22}\end{array}
 \begin{array}{c}
 a^2_{11}\\
 a^2_{21}\end{array}
 \right)
 \end{eqnarray*}
Thus the $2\times 2$ submatrices $A_{11}$ and $A_{21}$ are as
follows:
$$A_{11}= \left(\begin{array}{cc}
a^1_{11} & a^2_{11}\\
a^1_{21} & a^2_{21}\end{array}\right), \qquad
A_{21}=\left(\begin{array}{cc}
a^1_{12} & a^2_{11}\\
a^1_{22} & a^2_{21}\end{array}\right).$$

Condition~\ref{condition} at a vertex, say $v_{21}$, can be translated as
follows: since the facets $F^1_0, F^1_1$ and $F^2_0$ intersect at
$v_{21}$
\begin{eqnarray*}
\det\left(\begin{array}{c}
\mathbf e_1 \\
\mathbf a_1\\
\mathbf a_2\end{array}\right)
& = & \det\left(\begin{array}{ccc}
1 & 0 & 0 \\
a^1_{11} & a^1_{12} & a^2_{11} \\
a^1_{21} & a^1_{22} & a^2_{21}\end{array}\right)\\
& = & \det A_{21}=\pm1
\end{eqnarray*}
Similarly Condition~\ref{condition} at $v_{01}$ is equivalent to
$a^2_{21}=\pm1$, and that at $v_{20}$ is equivalent to $a^1_{12} =\pm1$.
These conditions are equivalent to
the condition that all principal minors of $A_{21}$ (including
the determinant of $A_{21}$ itself) are $\pm1$. Similarly
Condition~\ref{condition} at other vertices is equivalent to all principal
minors of $A_{11}$ being $\pm1$.
\end{example}

The last statement in Example~\ref{example:triangular cylinder} holds
in general.  A {\em principal minor}
of an $m\times m$ vector matrix $A$ of the form
(\ref{equation:matrix A}) means a principal minor
of an $m\times m$ matrix $A_{j_1\ldots j_m}$
for some $ 1\le j_1\le n_1,\ldots,1\le j_m\le n_m$
where the determinant of $A_{j_1\ldots j_m}$ itself is understood to be
a principal minor of $A_{j_1\ldots j_m}$.

\begin{lemma} \label{lemma:principal minor}
Let $P=\prod_{i=1}^m\Delta^{n_i}$.  If an $m\times m$ vector
matrix $A$ of the form (\ref{equation:matrix A}) is associated with
the characteristic function $\lambda$ of a quasitoric manifold over $P$, then
Condition~\ref{condition} for
$\lambda$ at all vertices of $P$ is equivalent to all principal minors of $A$
being $\pm 1$.
\end{lemma}

\begin{proof} The basic idea of the proof is same as in Example~\ref{example:triangular cylinder}.
Indeed, at a vertex $v_{j_1\ldots j_m}$ of $P$
all $n$ facets in $\mathcal F'=\mathcal F-\{F^i_{j_i}\mid i=1,\ldots m\}$
intersect. Hence Condition~\ref{condition} at $v_{j_1\ldots j_m}$ is equivalent to the determinant of the $n\times n$ matrix having $\lambda(F)$ as its row vectors for all $F\in\mathcal F'$ being $\pm 1$. But this determinant is nothing but a principal minor of the
$m\times m$ matrix $A_{j_1\ldots j_m}$ up to sign. Therefore the lemma follows.
\end{proof}

\begin{remark}\label{coro:diagonal vectors are 1}
It follows from the lemma above that each component $a^i_{ij}$ in the diagonal
entry vector $\mathbf a^i_i=(a^i_{i1},\ldots,a^i_{in_i})$ of the matrix $A$,
see (\ref{equation:matrix A}), is $\pm 1$ for $j=1,\ldots,n_i$.
The characteristic function $\lambda$ is defined up to sign and if we
change the sign of a vector $\lambda(F^j_k)$ in (\ref{basis})
(say $\lambda(F^j_k)=\mathbf e_\ell$), then the column
vector corresponding to $\lambda(F^j_k)$ (the $\ell$-th column)
changes the sign;
so we can always arrange $a^i_{i,j}=1$ for $i=1, \ldots, m$ and
$j=1, \ldots, n_i$, i.e., $\mathbf a^i_i=(1,\dots,1)$ by
an appropriate choice of signs of the vectors in (\ref{basis}).
In the following we always take $\mathbf a^i_i=(1,\dots,1)$ for $i=1,\dots,m$
for the matrix $A$ associated with a quasitoric manifold unless otherwise
stated.
\end{remark}

\section{Quotient construction}

It is known that any quasitoric manifold over a simple polytope
is realized
as the orbit space of the moment-angle manifold of the polytope
by some free torus action, see \cite{BP} and \cite{bu-pa-ra06}.
When the polytope is $\prod_{i=1}^m \Delta^{n_i}$,
the moment-angle manifold is the product $\prod_{i=1}^m S^{2n_i+1}$ of
odd dimensional spheres.  In this section we shall describe the
free torus action on it explicitly.
We remark that the case where $n_i=1$ for all $i$ (i.e., the polytope
is an $m$-cube) is treated in \cite{gr-ka94} and \cite{ci-ra05}.

\begin{lemma}\label{lemma:exponential equation}
If $C=(c_{ij})$ is a unimodular matrix of size $m$,
then the system of equations
$$z_1^{c_{i1}}\cdots z_m^{c_{im}}=1, \qquad  for \quad i=1,\ldots, m$$
has a unique solution $z_1=\cdots=z_m=1$ in $S^1\subset\mathbb C$.
\end{lemma}

\begin{proof}
Write $z_j=\exp(2\pi\theta_j\sqrt{-1})$ with $\theta_j\in \mathbb R$
for $j=1,\dots,m$. Then the equations in the lemma are equivalent to
\begin{equation*}\label{log equation}
c_{i1}\theta_1 + \cdots +c_{im}\theta_m=k_i
\qquad\textrm{ for }\quad i=1,\ldots, m
\end{equation*}
for some $k_i\in \mathbb Z$. Since $C$ is unimodular and $k_i$'s are integers,
$\theta_j$'s are also integers, which means $z_j=1$ for $j=1,\ldots, m$.
\end{proof}

Let $A$ be an $m\times m$ vector matrix in (\ref{equation:matrix A}). We
construct a quasitoric manifold $M(A)$ with $A$ as its corresponding
matrix. Consider the subspace $X=\prod_{i=1}^m S^{2n_i+1}$ of
$\prod_{i=1}^m\mathbb C^{n_i+1}$, which is the moment-angle manifold of
$\prod_{i=1}^m\Delta^{n_i}$.
Let $K=(S^1)^m$ and define an action of $K$ on $X$ by
\begin{eqnarray}\label{equation:K-action}
\lefteqn{(g_1,\ldots, g_m)\cdot((z^1_0,\ldots,z^1_{n_1}),\ldots, (z^m_0,\ldots,z^m_{n_m}))=}\\
&&((g_1z^1_0,(g_1^{a^1_{11}}\cdots g_m^{a^1_{m1}})z^1_1,\ldots, (g_1^{a^1_{1n_1}}\cdots g_m^{a^1_{mn_1}})z^1_{n_1}),
\ldots\nonumber\\
&&\ldots,(g_mz^m_0,(g_1^{a^m_{11}}\cdots g_m^{a^m_{m1}})z^m_1,\ldots, (g_1^{a^m_{1n_m}}\cdots g_m^{a^m_{m n_m}})z^m_{n_m}))\nonumber
\end{eqnarray}
where $(g_1,\ldots,g_m)\in K$ and $(z^i_0,\ldots, z^i_{n_i})\in S^{2n_i+1}\subset \mathbb C^{n_i+1}$ for $i=1,\dots,m$.

\begin{lemma}\label{lemma:free action of K}
The action of $K$ on $X$ defined in (\ref{equation:K-action}) is free if all principal minors of $A$ are equal to $\pm 1$.
\end{lemma}

\begin{proof}
To prove that the action is free we have to show that the equation
\begin{eqnarray}\label{equation:fixed point condition}
\lefteqn{(g_1,\ldots, g_m)\cdot((z^1_0,\ldots,z^1_{n_1}),\ldots, (z^m_0,\ldots,z^m_{n_m}))
}\\
& & =((z^1_0,\ldots,z^1_{n_1}),\ldots, (z^m_0,\ldots,z^m_{n_m}))\nonumber
\end{eqnarray}
implies $g_1=\cdots=g_m=1$.
Since $(z^i_0,\ldots,z^i_{n_i})\in S^{2n_i+1}$, at least one component,
say $z^i_{j_i}$, is nonzero for every $i=1,\ldots, m$.
If $z^i_0=0$ for all $i=1,\ldots, m$, then  equation (\ref{equation:fixed point condition})
implies that $g_1^{a^i_{1 j_i}}\cdots g_m^{a^i_{m j_i}}=1$ for all $i=1,\ldots, m$.
Since $\det A_{j_1\ldots j_m}=\pm1$ from the hypothesis, Lemma~\ref{lemma:exponential equation} implies that
$g_1=\cdots =g_m=1$.
Now suppose $z^i_0\ne 0$ for some $i=1,\ldots, m$.
For simplicity let us assume that there is some $0\le s\le m$ such that  $z^1_0=\cdots=z^s_0=0$
and $z^i_0\ne 0$ for all $i=s+1,\ldots, m$.
Then equation (\ref{equation:fixed point condition}) implies that
$g_1=\cdots =g_s=1$ and $g_{s+1}^{a^i_{(s+1) j_i}}\cdots g_m^{a^i_{m j_i}}=1$ for all $i=s+1,\ldots, m$.
Since all principal minors of $A_{j_1\ldots j_m}$ are $\pm1$, Lemma~\ref{lemma:exponential equation} implies that
$g_{s+1}=\cdots =g_m=1$, which proves the lemma.
\end{proof}

Since the action $K$ on $X$ is free, the orbit space $X/K$ is a smooth manifold of dimension $2n$.
Let $M(A)$ be the orbit space $X/K$ with the action of $G=(S^1)^n$ defined by
\begin{eqnarray}\label{G action on M(A)}
\lefteqn{(t_1,\ldots,t_n)\cdot [(z^1_0,\ldots,z^1_{n_1}),\ldots, (z^m_0,\ldots z^m_{n_m})]=}\\
&&[(z^1_0,t_1z^1_1\ldots,t_{n_1}z^1_{n_1}),\ldots, (z^m_0,t_{n-n_m+1}z^m_1\ldots t_nz^m_{n_m})].\nonumber
\end{eqnarray}
Then we have the following proposition.

\begin{proposition}\label{proposition:general quasitoric manifold}
$M(A)$ is a quasitoric manifold over $\prod_{i=1}^m\Delta^{n_i}$ with $A$
as its associated matrix.
\end{proposition}

\begin{proof}
We think of $q$-simplex $\Delta^q$ as
\[
\Delta^q=\{ (x_0,\dots,x_q)\in \mathbb R^{q+1}\mid x_0\ge 0,\dots,x_q\ge 0,
\sum_{i=0}^qx_i=1\}.
\]
Then
$P=\prod_{i=1}^m\Delta^{n_i}$ sits in $\prod_{i=1}^m \mathbb R^{n_i+1}$.
It is easy to see that $M(A)$ with the action of $G=(S^1)^n$ is
a quasitoric manifold over $P$ with the projection $\pi\colon M(A)\to P$
defined by
\[
\pi([(z^1_0,\ldots,z^1_{n_1}),\ldots (z^m_0,\ldots z^m_{n_m})])=
((|z^1_0|,\dots,|z^1_{n_1}|),\dots,(|z^m_0|,\dots,|z^m_{n_m}|)).
\]
The facets $F^i_j$ of $P$ are given by $x^i_j=0$ for some $1\le i\le m$ and
$0\le j\le n_i$, where $x^i_j$ denotes the $(j+1)$-st coordinate of the $i$-th
factor $\mathbb R^{n_i+1}$.  The isotropy subgroup of
a point in $\pi^{-1}(\interior F^i_j)$ is a circle subgroup. One can check
that it is the $(\sum_{k=1}^{i-1}n_k+j)$-th factor of $G=(S^1)^n$ when
$j\ge 1$ and the circle subgroup
\[
\{((g^{a^1_{i1}},\dots,g^{a^1_{in_1}}),\dots,(g^{a^m_{i1}},\dots,
g^{a^m_{in_m}}))\mid g\in S^1\}
\]
when $j=0$. This shows that if we denote the characteristic function of $M(A)$
by $\lambda$, then
\[
\lambda(F^1_1),\ldots,\lambda(F^1_{n_1}),\ldots,\lambda(F^m_1)\ldots,
\lambda(F^m_{n_m})
\]
are the standard basis elements of $\mathbb Z^n$
in the given order and
\[
\lambda(F^i_0)=((a^1_{i1},\dots,a^1_{in_1}),\dots,(a^m_{i1},\dots,
a^m_{in_m})) \in \mathbb Z^n \quad\text{for  $i=1,\ldots, m$,}
\]
which is the $i$-th row of our matrix $A$, proving the lemma.
\end{proof}

\section{Conjugation of vector matrices}

The correspondence between a quasitoric manifold
over $P=\prod_{i=1}^m\Delta^{n_i}$ and an $m\times m$ vector matrix $A$
depends on the order of the simplices $\Delta^{n_i}$'s in
the product formula of $P$. Namely,
if we consider $P=\prod_{i=1}^m\Delta^{n_{\sigma(i)}}$ for some permutation
$\sigma$ of $\{1,\ldots,m\}$, then the corresponding $m\times m$ vector matrix
$A_\sigma$ will be different from $A$.
In fact it is not difficult to see that if $E_\sigma$ is the $m\times m$
permutation matrix of $\sigma$ obtained from the identity matrix by permuting
the $i$-th row and column to $\sigma(i)$-th row and column respectively for
all $i=1,\ldots m$, then $A_\sigma = E_{\sigma}AE_\sigma^{-1}$.
One should be cautious that, as an $m\times m$ vector matrix, the entries in
the $j$-th column of $A_\sigma$ are vectors in
$\mathbb Z^{n_{\sigma(j)}}$ while the $j$-th column of $A$ are vectors in
$\mathbb Z^{n_j}$.

As an example let us consider $P$ as in Example~ \ref{example:triangular cylinder}.
If we consider $P=\Delta^1\times \Delta^2$ instead of $\Delta^2\times \Delta^1$ then the corresponding
$2\times 2$ vector matrix $A_\sigma$ is given by
\begin{eqnarray*}
A_\sigma & = & \left(\begin{array}{cc}
\mathbf a^2_2 & \mathbf a^1_2\\
\mathbf a^2_1 & \mathbf a^1_1
\end{array}\right)\\
& = & \left(\begin{array}{cc}
0 & 1 \\
1 & 0 \end{array}\right) A
\left(\begin{array}{cc}
0 & 1 \\
1 & 0\end{array}\right)^{-1}.
\end{eqnarray*}
The entries of the first column above are vectors in $\mathbb Z$
and the ones in the second column are in $\mathbb Z^2$.

We say that two $m\times m$ vector
matrices $A$ and $B$ are {\em conjugate} if there
exists an $m\times m$ permutation matrix $E_\sigma$ such that
$B = E_{\sigma}AE_\sigma^{-1}$.
In this case, the quasitoric manifolds $M(A)$ and $M(B)$ defined in
Proposition~\ref{proposition:general quasitoric manifold} are equivariantly
diffeomorphic.

Let $A$ be an $m\times m$ vector matrix of the form (\ref{equation:matrix A}).
A {\em proper principal minor} (resp. {\em determinant})
of $A$ means that a proper principal minor
(resp. determinant) of $A_{j_1\ldots j_m}$ for some $1\le j_1 \le n_1,\ldots,
1\le j_m\le n_m$.  The set of proper principal minors or determinants
is invariant under the conjugation relation.

\begin{lemma} \label{matrix}
Let $A$ be an $m\times m$ vector matrix of the form
(\ref{equation:matrix A}) such that all the proper principal
minors of $A$ are $1$.
If all the determinants of $A$ are $1$, then $A$ is conjugate to a
unipotent upper triangular vector matrix of the following form:
\begin{equation} \label{unipotent}
\left(\begin{array}{c c c c c}
\mathbf 1 & \mathbf  b_1^{2}& \mathbf b_1^{3} & \cdots  &\mathbf b_1^{m}\\
\mathbf 0 & \mathbf 1 & \mathbf b_2^{3}& \cdots  &\mathbf b_2^{m}\\
\vdots &&\ddots &\ddots&\vdots\\
\mathbf 0 & \cdots & \cdots & \mathbf 1 & \mathbf b_{m-1}^{m}\\
\mathbf 0 & \cdots & \cdots & \mathbf 0& \mathbf 1
\end{array}\right)
\end{equation}
where $\mathbf 0=(0,\ldots,0)$, $\mathbf 1=(1,\ldots, 1)$ of appropriate
sizes.
If all the determinants of $A$ are
$\pm 1$ and at least one of them is $-1$, then $A$ is conjugate to
a vector matrix of the following form:
\begin{equation} \label{matrixb}
\left(\begin{array}{c c c c c}
\mathbf 1 & \mathbf  b^2& \mathbf 0 & \cdots  &\mathbf 0\\
\mathbf 0 & \mathbf 1 & \mathbf b^3& \cdots  &\mathbf 0\\
\vdots &&\ddots &\ddots&\vdots\\
\mathbf 0 & \cdots & \cdots & \mathbf 1 & \mathbf b^m\\
\mathbf b^1 & \cdots & \cdots & \mathbf 0& \mathbf 1
\end{array}\right).
\end{equation}
where $\mathbf b^i$ is non-zero for any $i$
and $\prod_{i=1}^m b_i$, where $b_i$ is
any non-zero component of $\mathbf b^i$, is $(-1)^m2$.
(Therefore, the non-zero components in $\mathbf b^i$ are all same
for each $i$ and they are $\pm 1$ or $\pm 2$.)
\end{lemma}

\begin{proof}
The lemma is proved in \cite{MP} when $A$ is an ordinary
$m\times m$ matrix except the last statement on the components of
$\mathbf b^i$, and the proof for an $m\times m$ vector matrix is quite
similar. So we refer the reader to the cited paper and shall prove
only the statement on the components of $\mathbf b^i$.

Let $B$ be the vector matrix of the form (\ref{matrixb}).
The determinants of $A$ are $\pm 1$ and at least one of them is $-1$ by
assumption
while any determinant of $B$ is of the form $1+(-1)^{m+1}\prod_{i=1}^m b_i$
where $b_i$ is a component of $\mathbf b^i$.
Since the set of determinants of $A$ agrees with that of $B$ as remarked
above, it follows that there is a non-zero $b_i$ for each $i$ and
$\prod_{i=1}^m b_i=(-1)^m2$ whenever each $b_i$ is non-zero. This implies the
statement on $b_i$'s in the lemma.
\end{proof}

\section{\Higher Bott towers}

A quasitoric manifold over a product of simplices also
appears in iterated projective bundles.
For a complex vector bundle $E$, we denote the total space of its
projectivization by $P(E)$.

\begin{definition}\label{def:Bott manifold}
We call a sequence
\begin{equation}\label{equation:Bott tower}
\xymatrix{ B_m\ar[r]^-{\pi_m}& B_{m-1}\ar[r]^-{\pi_{m-1}}&
\cdots\ar[r]^-{\pi_2}& B_1\ar[r]^-{\pi_1}& B_0}=\{\textrm{a point}\},
\end{equation}
where $B_j=P(\mathbb C\oplus\xi_j)$ and $\xi_j$ is the Whitney sum of
complex line bundles over $B_{j-1}$, a
{\em \higher Bott tower} and each $B_j$ for $j=1,\dots,m$
a {\em \higher Bott manifold}.
\end{definition}

Each $B_j$ admits an effective action of $G_j=(S^1)^{\sum_{i=1}^j\dim\xi_i}$
defined as follows.  Assume
by induction that $B_{j-1}$ admits an effective action of $G_{j-1}$.
Then it lifts to an action on $\xi_j$ since $H^1(B_{j-1})=0$ although
the lifting is not unique, see \cite{ha-yo76}.
On the other hand since
$\xi_j$ is the Whitney sum of complex line bundles, it admits an
action of $(S^1)^{\dim \xi_j}$ by scalar multiplication on fibers.
These two actions commute and define an action of $G_j$ on $\xi_j$,
which induces an effective action of $G_j$ on $B_j$.
Without much difficulty it can be
shown that $B_j$ with the action of $G_j$ is a quasitoric
manifold over $\prod_{i=1}^j\Delta^{\dim\xi_i}$.
Furthermore each $B_j$ is a
nonsingular toric variety (i.e., a toric manifold).

\begin{proposition}\label{theorem:small covers are Bott manifolds}
Let $M$ be a quasitoric manifold over
$P=\prod_{i=1}^m\Delta^{n_i}$, and let $A$ be an $m\times m$ vector
matrix associated with $M$.
Then $M$ is equivalent to a \higher Bott manifold
if $A$ is conjugate to an $m\times m$ upper triangular
vector matrix of the form (\ref{unipotent}).
\end{proposition}

\begin{remark}  We will see later that the \lq\lq only if" statement
in the proposition above also holds, see
Lemma~\ref{matrix} and
Theorem~\ref{theorem:complex Bott manifold}.
\end{remark}

\begin{proof}
We may assume that $M=M(A)$ and $A$ is of the form (\ref{unipotent}).
We recall the quotient construction in Section~\ref{sect:3}.
Let $X_j=\prod_{i=1}^j S^{2n_i+1}$ for $j=1,\dots,m$, so $X_m$ agrees with
$X$ in
Section~\ref{sect:3}.  The group $K=(S^1)^m$ is acting on $X$ as in
(\ref{equation:K-action}) and $X/K=M(A)$.  We set $B_j=X_j/K$, so $B_m=M(A)$.
In the following we claim that the sequence
\[
\xymatrix{ B_m \ar[r]^-{\pi_m}& B_{m-1}\ar[r]^-{\pi_{m-1}}&
\cdots\ar[r]^-{\pi_2}& B_1\ar[r]^-{\pi_1}& B_0}=\{\textrm{a point}\}
\]
induced from the natural projections from $X_j$ on $X_{j-1}$ for
$j=m,\dots,2,1$ is a \higher Bott tower.

Since $A$ is of the form (\ref{unipotent}),
the last $(m-j)$ factors of $K=(S^1)^m$ are acting on $X_j$ trivially,
so the action of $K$ on $X_j$ reduces to an action of the product $K_j$
of the first $j$ factors of $K=(S^1)^m$.
This means that $X_j/K=X_j/K_j$.
Moreover, the last factor of $K_j$
is acting on the last factor $S^{2n_j+1}$ of $X_j$ as scalar multiplication
and trivially on the other factors of $X_j$. Therefore
the map $\pi_j\colon B_j=X_j/K_j \to B_{j-1}=X_{j-1}/K_{j-1}$ is
a fibration with $\mathbb CP^{n_j}=S^{2n_j+1}/S^1$ as a fiber and
this is actually the projectivization of a complex vector bundle $\xi_j$ over
$B_{j-1}$.  In fact, the bundle $\xi_j$ is obtained as follows.
Let $V_j$ be $\mathbb C^{n_j+1}$ with the linear $K_{j-1}$-action defined by
\[
\begin{split}
&(g_1,\ldots, g_{j-1})\cdot(z^j_0,\ldots,z^j_{n_j})\\
=&(z^j_0,(g_1^{b^j_{11}}\cdots g_{j-1}^{b^j_{j-1\ 1}})z^j_1,\ldots,
(g_1^{b^j_{1n_j}}\cdots g_{j-1}^{b^j_{j-1\ n_j}})z^j_{n_j})
\end{split}
\]
where $\mathbf b^j_i=(b^j_{i1},\dots,b^j_{i n_j})$ is a vector
in \eqref{unipotent} for $i=1,\dots,j-1$.
Since the action of $K_{j-1}$ on $X_{j-1}$ is free, the projection
$$(X_{j-1}\times V_j)/K_{j-1}\to X_{j-1}/K_{j-1}
=B_{j-1}$$
becomes a vector bundle, where the action of $K_{j-1}$ on
$X_{j-1}\times V_j$ is a diagonal one.
This is the desired bundle $\xi_j$ and since
$V_j$ decomposes into sum of complex one dimensional $K$-modules, the bundle
$\xi_j$ decomposes into the Whitney sum of complex line bundles accordingly.
\end{proof}

One can describe the bundles $\xi_j$ in the proof of the proposition above
more explicitly. For that let us fix some notation.
For a vector bundle $\eta$ and a vector $\mathbf a=(a_1, \ldots, a_n)\in \mathbb Z^n$
let $\eta^{\mathbf a}$ denote the bundle $\eta^{a_1}\oplus\cdots\oplus \eta^{a_n}$. For vector bundles
$\eta_1, \ldots,\eta_k$ over a space and vectors $\mathbf a_1=(a_{11},\ldots, a_{1n}), \ldots, \mathbf a_k=(a_{k1},\ldots ,a_{kn})$
let
\begin{eqnarray*}
\odot_{i=1}^k\eta_i^{\mathbf a_i}&=&
\eta_1^{\mathbf a_1}\odot\cdots\odot\eta_k^{\mathbf a_k}\\
&=&(\eta^{a_{11}}_1\otimes\cdots\otimes\eta_k^{a_{k1}})\oplus\cdots\oplus(\eta^{a_{1n}}_1\otimes\cdots\otimes\eta_k^{a_{km}})
\end{eqnarray*}
where the last expression denotes the Whitney
sum of componentwise tensor products.

Let $\xi_{1}^{2}$ denote the canonical line bundle over $B_1$ and let
$\xi_{1}^{3}=\pi^\ast_{2}(\xi_{1}^{2})$ the pull-back bundle of the
canonical line bundle over $B_1$ to $B_2$ via the projection
$\pi_2\colon B_2\to B_1$. In general, let
$\xi_{j-1}^{j}$ be the canonical line bundle over $B_{j-1}$, and we
inductively define
$$\xi_{j-k}^{j}=\pi^\ast_j\circ\cdots\circ\pi^\ast_{j-k+1}
(\xi_{j-k}^{j-k+1})\quad\text{for $k=2,\dots,j-1$}.
$$
Then one can see that $\xi_j=\odot_{i=1}^{j-1}(\xi_i^j)^{\mathbf b_{i}^{j}}$.

A \higher Bott manifold is not only a quasitoric manifold over a product of
simplices but also a complex manifold on which the
action preserves the complex structure, in particular,
it has an almost complex structure left invariant under the action.
The following theorem shows that the converse holds.
We remark that the equivalence (1) $\Leftrightarrow$ (3) is
a particular case of \cite[Theorem 6]{Do}.

\begin{theorem}\label{theorem:complex Bott manifold}
Let $M$ be a quasitoric manifold over
$P=\prod_{i=1}^m\Delta^{n_i}$, and let $A$ be the
$m\times m$ vector matrix associated with $M$ which has
$\mathbf 1$ as the diagonal entries.
Then the following are equivalent:
\begin{enumerate}
\item $M$ is equivalent to a \higher Bott manifold.
\item $M$ is equivalent to a quasitoric manifold which
admits an invariant almost complex structure under the action.
\item All the principal minors of $A$ are $1$.
\end{enumerate}
\end{theorem}

\begin{proof}
The implication $(1)\Rightarrow (2)$ is obvious and the implication
$(3)\Rightarrow (1)$ follows from
Proposition~\ref{theorem:small covers are Bott manifolds} and
Lemma~\ref{matrix}, so
it suffices to prove the implication $(2)\Rightarrow (3)$.

We may assume that $M$ itself admits an invariant almost complex structure.
As is noted in the paragraph before Condition~\ref{condition} we can
define a sign-unambiguous characteristic function $\lambda$ of $M$.
Let $\Lambda$ be the matrix associated with
$\lambda$. To each cubical face of $P$, the submanifold of $M$ over it
inherits an invariant almost complex structure, so it follows from
\cite[Theorem 3.4]{MP} that all principal minors of the restriction of
$-\Lambda$
to each cubical face of $P$ are equal to 1.  Therefore $A=-\Lambda$ and
this proves (3).
\end{proof}

\begin{remark} \label{difference}
A difference between quasitoric manifolds and small covers
appears here.  Namely, not every quasitoric manifold over a product
of simplices is equivalent to a \higher Bott manifold as is seen from
Theorem~\ref{theorem:complex Bott manifold},
while it follows from the real version of
Proposition~\ref{theorem:small covers are Bott manifolds} and the
$\Z/2$ version of the former part of Lemma~\ref{matrix}
that every small cover over a product of simplices turns out to be
equivalent to a \higher Bott manifold (over $\R$).
\end{remark}

\section{Cohomology ring} \label{sect:cohom}

The connected sum $\C P^2\# \C P^2$ is a quasitoric manifold over
a square but not homeomorphic to a Bott manifold (or Hirzebruch surface)
over a square.  In the rest of this paper, we shall give
a sufficient condition in terms of cohomology ring
for a quasitoric manifold over a product of simplices
to be homeomorphic to a \higher Bott manifold (Theorem~\ref{cotri}).
This section is a preliminary section for the purpose.

\begin{lemma}\label{lemma:cohomology ring of quasitoric and small cover}
Let $M$ be a quasitoric manifold
over $\prod_{i=1}^m\Delta^{n_i}$ and let $A$ be the vector matrix
of the form (\ref{equation:matrix A}) associated with $M$. Then
\begin{equation} \label{cohoM}
H^{\ast}(M)=\Z[y_1,\ldots, y_{m}]/L
\end{equation}
where the ideal $L$ is generated by the following $m$ expressions:
\begin{equation} \label{yrela}
y_k\cdot\prod_{j=1}^{n_k}(\sum_{i=1}^m a^k_{ij}y_i)\quad
\text{for $k=1,\dots,m$.}
\end{equation}
\end{lemma}

\begin{proof}
We will use the result (\ref{equation:cohomology ring}).
In our case, the matrix in (\ref{equation:lambda}) is of the form
\begin{equation} \label{equation:AI}
(\lambda_{ij})=\left(\begin{array}{c}
A\\
I_n
\end{array}\right)
\end{equation}
where $I_n$ is the $n\times n$ identity matrix.
Let
$$\omega^1_0,\ldots\omega^1_{n_1},\ldots,\omega^m_0,\ldots,\omega^m_{n_m}$$
be the indeterminates corresponding to the facets
$$F^1_0,\ldots,F^1_{n_1},\ldots, F^m_0,\ldots, F^m_{n_m}$$
in the given order.
Then by
(\ref{equation:cohomology ring}) we have
\begin{equation}\label{equation:cohomology ring2}
H^{\ast}(M)\cong
 \Z[\omega^1_0,\ldots,\omega^1_{n_1},\ldots,\omega^m_0,\ldots,
 \omega^m_{n_m}]/(I+J)
\end{equation}
where $I$ is the ideal generated by the monomials
$$\omega^i_0\cdots\omega^i_{n_i}\qquad \text{for }i=1,\ldots,m$$
because the intersection of facets $F^i_0,\ldots,F^i_{n_i}$ is empty
for $i=1,\ldots, m$, and $J$ is the ideal generated by
\[
\begin{split}
\lambda_j= &\ \ \lambda_{1j}\omega^1_0+\cdots+\lambda_{mj}\omega^m_0\\
&+\lambda_{(m+1)j}\omega^1_1+\cdots+\lambda_{(m+n_1)j}\omega^1_{n_1}\\
&+\cdots\\
&+\lambda_{(m+\sum_{i=1}^{m-1}n_i+1)j}\omega^m_1+\cdots
\lambda_{(m+n)j}\omega^m_{n_m}
\end{split}
\]
for $j=1,\ldots, m+n$ because the order of the row vectors in
(\ref{equation:AI}) is
\[
\lambda(F^1_0),\dots,\lambda(F^m_0),\lambda(F^1_1),\dots,\lambda(F^1_{n_1}),
\dots, \lambda(F^m_1),\dots,\lambda(F^m_{n_m}).
\]

If $j=(\sum_{i=1}^{k-1}n_i)+\ell$ and $1\le \ell \le n_k$,
then
$$\lambda_j=a^k_{1\ell}\omega^1_0+a^k_{2\ell}\omega^2_0 + \cdots +
a^k_{m\ell}\omega^m_0+\omega^k_\ell$$
Since $\lambda_j=0$ in $H^\ast(M)$,
we have that
\begin{equation} \label{omega}
\omega^k_\ell=-(a^k_{1\ell}\omega^1_0+a^k_{2\ell}\omega^2_0 + \cdots +
a^k_{m\ell}\omega^m_0).
\end{equation}
Set $y_k=\omega^k_0$ for $k=1,\ldots, m$. Then
$\omega^k_0\cdots\omega^k_{n_1}=0$ in the cohomology ring implies that
\[
y_k\prod_{\ell=1}^{n_k}(a^k_{1\ell}y_1+a^k_{2\ell}y_2+\cdots+a^k_{m\ell}y_m)=0.
\]
This proves the relation in the lemma.
\end{proof}

\begin{lemma}\label{lamma:alpha_{ij}is not zero}
Let $M$ and $y_1,\ldots, y_m$ be as above.
Let $x=\sum_{j=1}^mb_jy_j$ be an element of $H^\ast(M)$ such that
$b_j\ne 0$ for some $j$. Then $x^{n_j}\ne 0$ in $H^\ast(M)$.
\end{lemma}

\begin{proof}
Suppose $x^{n_j}=0$ on the contrary. Then $(\sum_{j=1}^mb_jy_j)^{n_j}$
must be in the ideal $L$ in (\ref{yrela}).
However, $y_j^{n_j+1}$ is the least power of $y_j$ which appears
as a term in a polynomial of $L$ while $(\sum_{j=1}^mb_jy_j)^{n_j}$  contains
a non-zero scalar multiple of $y_j^{n_j}$ because $b_j\not=0$ by assumption.
This is a contradiction.
\end{proof}

\begin{lemma} \label{restr}
Let $M(j)$ be a facial submanifold of $M$ over
$\prod_{i\not=j}^m\Delta^{n_i}$.
Then $H^*(M(j))$ is equal to (\ref{cohoM}) with $y_j=0$ plugged in.
\end{lemma}

\begin{proof}
Let $y_1,\dots,y_m$ be the generators of $H^*(M)$ in
Lemma~\ref{lemma:cohomology ring of quasitoric and small cover}.
We may assume that $M(j)$ is over $\prod_{i\not=j}^m\Delta^{n_i}\times \{v\}$
where $v$ is a vertex of $\Delta^{n_j}$ and also that $y_j$ is the dual of
the characteristic submanifold $M_j$ over
$\prod_{i\not=j}^m\Delta^{n_i}\times \Delta^{n_j-1}(v)$ where
$\Delta^{n_j-1}(v)$ is the facet
of $\Delta^{n_j}$ not containing $v$.  Since $M(j)$ and $M_j$
have no intersection, the restriction of $y_j$ to $M(j)$ vanishes.

We know that
\begin{equation} \label{HM}
H^*(M)=\Z[y_1,\dots,y_m]/(g_1,\dots,g_m).
\end{equation}
where $g_k$ is the polynomial in (\ref{yrela}).
Since $y_j$ maps to zero in $H^*(M(j))$ and $g_j$ contains $y_j$ as
a factor, we have a natural surjective map
\[
\Z[y_1,\dots,\widehat{y_j},\dots,y_m]/(g_1',
\dots,\widehat{g_j'},\dots,g_m')
\to H^*(M(j)).
\]
where $g_k'$ denotes $g_k$ with $y_j=0$ plugged in and $\widehat{\ }$ denotes
the term there is dropped.
The degree of $g_k'$ for $k\not=j$ is $n_k+1$ and
$g_k'$ contains the term $y_k^{n_k+1}$.
Therefore, the ranks of the both sides above agree, so that the
map is an isomorphism.  This proves the lemma.
\end{proof}

\begin{lemma} \label{H2nonzero}
Let $N$ be the smallest number among $n_i$'s.
If the vector matrix associated with $M$ is of the form (\ref{matrixb})
in Lemma~\ref{matrix}, then
there is no non-zero element in $H^2(M)$ whose $(N+1)$-st power vanishes.
\end{lemma}

\begin{proof}
Let $y$ be an element of $H^2(M)$ whose $(N+1)$-st power vanishes.
Since $N$ is smallest among $n_i$'s,
$y$ can be expressed as a linear combination of the canonical generators
$y_i$'s with $n_i=N$
by Lemma~\ref{lamma:alpha_{ij}is not zero},
say $y=\sum_{n_i=N}{a_iy_i}$ with $a_i\in \Z$.
All relations in $H^*(M)$ of cohomological degree
$2(N+1)$ are generated by $y_i^{k_i+1}(y_i+b_{i}y_{i-1})^{n_i-k_i}$'s
with $n_i=N$ over $\Z$, where
$y_{i-1}$ with $i=1$ is understood to be $y_m$, $b_i$ is the non-zero
component of the vector $\mathbf b_i$ in Lemma~\ref{matrix}
and $k_i$ is the number of zero components of $\mathbf b_i$.
Note that $k_i<N$ when
$n_i=N$ since $\mathbf b_i$ is non-zero.  It follows that
we obtain a polynomial identity
\begin{equation} \label{excep}
(\sum_{n_i=N}{a_iy_i})^{N+1}=\sum_{n_i=N}a_i^{N+1}y_i^{k_i+1}
(y_i+b_{i}y_{i-1})^{N-k_i}.
\end{equation}

{\it Case 1.} The case where $N=1$. In this case $k_i=0$ for $i$ with
$n_i=N=1$.  Suppose that
$a_i$ is non-zero for some $i$ with $n_i=1$.  Comparing the coefficients
of $y_i^2$ and $y_iy_{i-1}$
at both sides of the identity (\ref{excep}) with an observation that the
right-hand side of (\ref{excep}) contains a $y_{i}y_{i-1}$-term,
we see that $n_{i-1}=1$ and $2a_ia_{i-1}=a_i^2b_i$.
Since $a_i$ and $b_i$ are both non-zero, this shows that $a_{i-1}$ is also
non-zero and $2a_{i-1}=a_ib_i$.
Since $n_{i-1}=1$ and $a_{i-1}$ is non-zero, the same argument
can be applied to $i-1$ instead of $i$. Repeating this argument,
we see that $n_i=1$ and $2a_{i-1}=a_ib_i$ for any $i$.
It follows that $\prod_{i=1}^m {b_i}=2^m$ which
contradicts the fact that $\prod_{i=1}^m b_i=(-1)^m2$ in Lemma~\ref{matrix}.

{\it Case 2.} The case where $N\ge 2$.
When we expand the right hand side of the identity (\ref{excep}), no
monomial in more than two variables appears.  Since $N\ge 2$, this implies
that at most two coefficients among $a_i$'s are non-zero.
Since all $b_i$'s are non-zero, it easily follows from (\ref{excep}) that
the case where only one coefficient among $a_i$'s is non-zero does not
occur.

Suppose that there
are exactly two non-zero coefficients, say $a_i$ and $a_j$.
Then only two variables appear at the left hand side.   Unless $m=2$ and
$n_1=n_2=N$, at least three variables
appear at the right hand side of (\ref{excep}) which is a contradiction.
If $m=2$ and $n_1=n_2=N$, then the identity (\ref{excep}) is
\[
(a_1y_1+a_2y_2)^{N+1}=a_1^{N+1}y_1^{k_1+1}(y_1+b_1y_2)^{N-k_1}+
a_2^{N+1}y_2^{k_2+1}(y_2+b_2y_1)^{N-k_2}.
\]
Replacing $y_2$ by $-b_2y_1$ above, we obtain an identity
$$|a_1-a_2b_2|^{N+1}=|a_1|^{N+1}$$
where we used the fact $b_1b_2=2$ in
Lemma~\ref{matrix}.
Since $a_2b_2\not=0$, it follows from the identity above that $2a_1=a_2b_2$.
Similarly, replacing $y_1$ by $-b_1y_2$ above,
we obtain $2a_2=a_1b_1$.  These two
identities imply that $b_1b_2=4$ which contradicts to $b_1b_2=2$.

This completes the proof of the lemma.
\end{proof}

\section{Cohomologically product quasitoric manifolds}

We say that a quasitoric manifold $M$
over $\prod_{i=1}^m\Delta^{n_i}$ is \emph{cohomologically product} over
$\Q$ if there are elements $x_1,\dots,x_m$ in $H^2(M;\Q)$ such that
\begin{equation}\label{eq:cohomology of quasitor which is product of CP^n}
H^\ast(M;\Q)=\Q[x_1,\ldots,x_m]/(x_1^{n_1+1},\dots, x_m^{n_m+1}).
\end{equation}
The purpose of this section is to prove the following.

\begin{theorem} \label{cotri}
If a quasitoric manifold $M$ over $\prod_{i=1}^m\Delta^{n_i}$ is
cohomologically
product over $\Q$, then the vector matrix associated with $M$ is conjugate to
a unipotent upper triangular vector matrix, so that $M$ is homeomorphic to
a \higher Bott manifold.
\end{theorem}

\begin{remark}
   We prove in \cite{ch-ma-su07} that if a generalized Bott manifold is cohomologically trivial over $\Z$, then it is diffeomorphic to a product of complex projective spaces. This together with Theorem \ref{cotri} implies that if a quasitoric manifold over a product of simplices is cohomologically trivial over $\Z$, then it is homeomorphic to a product of complex projective spaces.
\end{remark}

In the following $M$ is assumed to be cohomologically product over $\Q$.
We have another set of generators $\{y_1,\dots,y_m\}$
in Lemma~\ref{lemma:cohomology ring of quasitoric and small cover}.
Since both $\{x_1,\ldots,x_m\}$ and $\{y_1,\ldots, y_m\}$ are sets of
generators of $H^2(M;\Q)$, one can write
\begin{equation} \label{matrixC}
y_j=\sum_{i=1}^m c_{ji}x_i \qquad\textrm{ for }j=1,\ldots, m\textrm{\  and  }
c_{ji}\in\Q,
\end{equation}
where the coefficient matrix $C=(c_{ji})$ has non-zero
determinant.

\begin{lemma} \label{cii}
By an appropriate change of indices in $x_i$'s and $y_j$'s,
we may assume that $c_{jj}\not=0$ for any $j=1,\dots,m$.
\end{lemma}

\begin{proof}
We may assume that $n_1\ge n_2\ge \ldots\ge n_m$ by an appropriate change
of indices.   Let $S=\{N_1,\ldots, N_k\}$ be the set of all
distinct elements of $n_1,\ldots,n_m$ such that $N_1>\ldots>N_k$.
We can view $\{n_1,\ldots,n_m\}$ as a function
$\mu:\{1,\ldots,m\}\to \mathbb N$ such that $\mu(j)=n_j$. Then $S$
is the image of $\mu$. Let $J_\ell=\mu^{-1}(N_\ell)$ for $\ell=1,\ldots, k$.
We write
\begin{equation}\label{eq:block form of relation between x and y}
x_i=\sum_{j=1}^md_{ij}y_j
\qquad\textrm{ for }i=1,\ldots, m\textrm{\  and  }d_{ij}\in\Q.
\end{equation}
Since $x_i^{n_i+1}=0$,
$d_{ij}=0$ if $n_i< n_j$ by Lemma~\ref{lamma:alpha_{ij}is not zero}.
This shows that $D=(d_{ij})$ is a block upper triangular matrix
because we assume $n_1\ge n_2\ge \ldots\ge n_m$.
The matrix $C$ in (\ref{matrixC}) is the inverse of the matrix $D$, so
$C$ is also a block upper triangular matrix and of the same type as $D$,
i.e.,
$$C=\left(
\begin{array}{cccc} C_{J_1} & && \ast \\ & C_{J_2} & & \\ & & \ddots& \\ 0
&&&C_{J_k}
\end{array}\right)
$$
where $C_{J_\ell}$ $(\ell=1,\dots,k)$ is a square matrix formed from
$c_{ij}$ with $i,j\in J_\ell$.
Since $\det C\not=0$, we have $\det C_{J_\ell}\not=0$ for any $\ell$.
By definition of determinant $\det
C_{J_\ell}=\sum_{\sigma}{\sgn\sigma}\prod_{j\in J_{\ell}}c_{j\sigma(j)}$
where the sum is taken over all permutations $\sigma$ on $J_\ell$.
Therefore there must exist a permutation $\sigma$ on $J_\ell$ such that
$\prod_{j\in J_\ell}c_{j\sigma(j)}\ne 0$.
This implies the lemma.
\end{proof}

\begin{lemma} \label{cpror}
The facial submanifold $M(j)$ of $M$ over $\prod_{i\not=j}^m\Delta^{n_i}$
is also cohomologically product over $\Q$ for any $j$.
\end{lemma}

\begin{proof}
Since $H^*(M(j))$ is $H^*(M)$ with $y_j=0$ plugged by Lemma~\ref{restr},
it follows from (\ref{matrixC}) that
\[
H^*(M(j);\Q)=\Q[x_1,\dots,x_m]/(x_1^{n_1+1},\dots,x_m^{n_m+1},
\sum_{i=1}^mc_{ji}x_i).
\]
Here $c_{jj}\not=0$ by Lemma~\ref{cii}, so that one can
eliminate the variable $x_j$ using the
relation $\sum_{i=1}^m c_{ji}x_i=0$.  Therefore a natural map
\[
\Q[x_1,\dots,\widehat{x_j},\dots,x_m]/
(x_1^{n_1+1},\dots,\widehat{x_j^{n_r+1}},\dots,x_m^{n_m+1})
\to H^*(M(j);\Q)
\]
is surjective.  Since the dimensions at the both sides above are same,
this map is actually an isomorphism, proving the lemma.
\end{proof}

Now we shall prove Theorem~\ref{cotri} by induction on the number $m$ of
factors in $\prod_{i=1}^m\Delta^{n_i}$.
Suppose that $M$ is cohomologically product over $\Q$.
Then any facial submanifold $M(j)$ is cohomologically product over $\Q$
by Lemma~\ref{cpror}. Therefore by induction assumption
all the proper principal
minors of the vector matrix $A$ associated with $M$ are $1$.
It follows that the vector matrix $A$ is conjugate
to a unipotent upper triangular vector matrix or
to a matrix of the form (\ref{matrixb}) in Lemma~\ref{matrix}.
But the latter does not occur because since
$M$ is cohomologically product over $\Q$,
$H^2(M)$ must contain a non-zero element
whose $(N+1)$-st power vanishes, where $N$ is the smallest number
among $n_j$'s, but this fact contradicts Lemma~\ref{H2nonzero}.
This proves Theorem~\ref{cotri}.

\end{document}